\documentclass[a4paper,12pt]{article}
\usepackage{amsmath}
\usepackage{amssymb}
\usepackage{amsthm}
\usepackage[english]{babel}

\newtheorem{theor}{Theorem}[section]
\newtheorem{cor}[theor]{Corollary}
\newtheorem{defi}[theor]{Definition}
\newtheorem{prop}[theor]{Lemma}

\newcommand{\Irr}{\mathrm{Irr}}

\begin{document}

\begin{center}
{\Large 
Vanishing classes for $p$-singular characters of symmetric groups}

Lucia Morotti
\end{center}

\begin{abstract}
We call an irreducible character $p$-singular if $p$ divides its degree. We prove a number of equivalent conditions for a character of the symmetric group $S_n$ to be $p$-singular, involving a certain family of conjugacy classes. This answers in part a question by Navarro.
\end{abstract}

\section{Introduction}

The work presented here started from the following question of Navarro to Olsson (December 2010):

\noindent ``If $p$ is a prime, what are the elements $x$ of the symmetric group $S_n$ such that $\chi(x)=0$ for all $\chi\in\Irr(S_n)$ of degree divisible by $p$?''

We will now give two definitions which will allow us to reformulate the above question.

\begin{defi}[$p$-singular character]
Let $\chi$ be an irreducible character of a finite group and let $p$ be prime. We say that $\chi$ is \emph{$p$-singular} if $p$ divides its degree.
\end{defi}

\begin{defi}[$p$-vanishing class]
A conjugacy class of a finite group $G$ is called \emph{$p$-vanishing} if all $p$-singular irreducible characters of $G$ take value 0 on that conjugacy class.
\end{defi}

Let $p$ be a prime and $n$ a non-negative integer. The above question can then be reformulated as: what are the $p$-vanishing conjugacy classes of $S_n$?

In this paper we will give a partial answer to this question, showing that certain conjugacy classes of $S_n$ are $p$-vanishing (see Corollary \ref{c1}). For $p\geq 5$ we conjecture that this answers the question completely.

Through all of the paper let $n=a_0+a_1p+\ldots+a_kp^k$ be the $p$-adic decomposition of $n$. In order to simplify notations in the following we define $a_i:=0$ for $i>k$. We will start by giving a name to a family of partitions which will play an important role in the paper. In the following, for $m\geq 0$, we will write $\beta\vdash m$ to indicate that $\beta$ is a partition of $m$.

\begin{defi}[Partition of $p$-adic type]
A partition of $n$ is of \emph{$p$-adic type} if it is of the form
\[\left(d_{k,1}p^k,\ldots,d_{k,{h_k}}p^k,\ldots,d_{0,1},\ldots,d_{0,{h_0}}\right)\]
with $(d_{i,1},\ldots,d_{i,{h_i}})\vdash a_i$ for $0\leq i\leq k$.
\end{defi}

As $0\leq a_i<p$ for $0\leq i\leq k$ we have that a partition $\alpha=(\alpha_j)_{j\geq 0}$ is of $p$-adic type if and only if
\[\sum_{{j:p^i\mid\alpha_j,}\atop{p^{i+1}\!\not\,|\,\alpha_j}}\alpha_j=a_ip^i\]
for $0\leq i\leq k$. For example the partition $\lambda_n:=((p^k)^{a_k},\ldots,p^{a_1},1^{a_0})$ (written in exponential notation) is of $p$-adic type. This partition will have a special role in the following.

For $\alpha$ and $\beta$ partitions of $n$, let $\chi^\alpha$ be the irreducible character of $S_n$ labeled by $\alpha$ and $\chi^\alpha_\beta$ be the value that $\chi^\alpha$ takes on the conjugacy class of $S_n$ labeled by $\beta$. We are now ready to state our main result, which gives a number of equivalent conditions for an irreducible character of $S_n$ to be $p$-singular.

\begin{theor}\label{t2}
Let $\alpha\vdash n$. The following are equivalent.
\begin{enumerate}
\def\theenumi{(\roman{enumi})}
\renewcommand{\labelenumi}{(\roman{enumi})}
\item\label{i1}
$\chi^\alpha$ is $p$-singular.

\item\label{i3}
$\chi^\alpha_\beta=0$ for every $\beta\vdash n$ of $p$-adic type.

\item\label{i2}
$\chi^\alpha_{\lambda_n}=0$.

\item\label{i4}
We cannot remove from $\alpha$ a sequence of hooks of lengths given, in order, by the parts of $\lambda_n$.
\end{enumerate}
\end{theor}

The following is an easy corollary to the above theorem and generalises Theorem 4.1 of \cite{b5}.

\begin{cor}\label{c1}
Conjugacy classes of $S_n$ labeled by partitions of $p$-adic type are $p$-vanishing.
\end{cor}

In Section \ref{s1} we will prove Theorem \ref{t2} by proving the following chain of implications:
\[\ref{i1}\Rightarrow\ref{i3}\Rightarrow\ref{i2}\Rightarrow\ref{i4}\Rightarrow\ref{i1}.\]
In \cite{b5}, Malle, Navarro and Olsson already proved some implications between the conditions in Theorem \ref{t2}. They proved that $\ref{i1}\Rightarrow\ref{i4}\Rightarrow\ref{i2}$, where $\ref{i4}\Rightarrow\ref{i2}$ is a consequence of the Murnaghan-Nakayama formula.

\section{Preliminary Lemmas}

In this section we will state and give references for or prove all the results needed to prove Theorem \ref{t2}. In the following, for $h\geq 1$, let $w_h(\alpha)$ be the $h$-weight of $\alpha$, that is the maximum number of $h$-hooks which can be recursively removed from $\alpha$. The $h$-weight of a partition is also equal to the number of its hooks of length divisible by $h$ (see Proposition 3.6 of \cite{b4}). For more reference about hooks, hook lengths, hook removal and $h$-weight, as well as about $h$-cores, see Sections I.1 and I.3 of \cite{b4}.

\begin{defi}\label{d1}
Let $\alpha$ be a partition of $n$. For $i\geq 0$ define
\[b_i(\alpha):=w_{p^i}(\alpha)-pw_{p^{i+1}}(\alpha).\]
\end{defi}

Lemma 4.4 of \cite{b5} says that if $\beta$ and $\gamma$ are partitions such that $\gamma$ is obtained from $\beta$ by removing a hook of length $hl$, then $w_h(\gamma)=w_h(\beta)-l$. Repeatedly applying it we easily have the following.

\begin{prop}\label{p1}
Let $\alpha$ and $\beta$ be partitions such that $\beta$ is obtained from $\alpha$ by removing a sequence of hooks of lengths $h\gamma_1,\ldots,h\gamma_m$. Then
\[w_h(\beta)=w_h(\alpha)-\sum_{i=1}^m\gamma_i.\]
\end{prop}

For $h\geq 1$ let $\alpha_{(h)}$ be the $h$-core of $\alpha$, that is the (unique) partition obtained from $\alpha$ by recursively removing all the $h$-hooks (for a proof of the uniqueness see Proposition 3.2 of \cite{b4}). We then have by the previous lemma and the definitions of $p^i$-weight and $p^i$-core, that $b_i(\alpha)=w_{p^i}(\alpha_{(p^{i+1})})$ for all $i$. In particular $b_i(\alpha)\geq 0$ and the following lemma holds (see Proposition 4.5 of \cite{b5} for a proof).

\begin{prop}\label{l'9}
For $j\geq 0$ we have that
\[w_{p^j}(\alpha)=\sum_{i\geq j}p^{i-j}b_i(\alpha).\]
\end{prop}

The following lemma is Proposition 4.6 of \cite{b5}.

\begin{prop}\label{l1}
We have that $b_k(\alpha)\leq a_k$.

If $b_k(\alpha)=a_k$ then it is possible to remove $a_k$ hooks of length $p^k$ from $\alpha$. The resulting partition $\beta$ is unique and it satisfies $b_k(\beta)=0$ and $b_i(\beta)=b_i(\alpha)$ for $0\leq i<k$.
\end{prop}

In the proof of Theorem \ref{t2} we will need a corollary to this lemma.

\begin{cor}\label{c3}
Assume that $b_i(\alpha)\not=a_i$ for some $i\geq 0$. Then there exists $j$ maximal such that $b_j(\alpha)\not=a_j$. For this $j$ we have $j\leq k$ and $b_j(\alpha)<a_j$.
\end{cor}

\begin{proof}
Notice that since $b_i(\alpha)=0=a_i$ for $i>k$ as $n<p^{k+1}$, we have by assumption that $j$ exists and that $j\leq k$. The corollary follows then from repeated application of Lemma \ref{l1}, as if $b_k(\alpha)=a_k$ and $\beta$ is as in the text of Lemma \ref{l1} then $|\beta|=a_0+\ldots+a_{k-1}p^{k-1}$.
\end{proof}

The next lemma allows us to tell whether $\chi^\alpha$ is $p$-singular or not by just looking at $a_i$ and $b_i(\alpha)$ for all $i$. A proof of it can be found in Sections 3 and 4 of \cite{m2}.

\begin{prop}\label{c2}
We have that $\chi^\alpha$ is not $p$-singular if and only if $b_i(\alpha)=a_i$ for every $i\geq 0$.
\end{prop}

We still need a lemma before proving Theorem \ref{t2}. The proof follows from Corollary 2.7.33 of \cite{b1}.

\begin{prop}\label{p2}
Assume that $h\geq 1$ and that $\beta$ is a partition of $n$ with exactly $w_h(\alpha)$ parts of length $h$ and let $\gamma\vdash n-hw_h(\alpha)$ be obtained from $\beta$ by removing the parts of length $h$. Then
\[\chi^\alpha_\beta=c^\alpha_\beta\chi^{\alpha_{(h)}}_\gamma\]
with $c^\alpha_\beta\not=0$.
\end{prop}

\section{Proof of Theorem \ref{t2}}\label{s1}

We are now ready to prove Theorem \ref{t2}.

{\bf ``\ref{i1}$\mathbf{\Rightarrow}$\ref{i3}''}: Assume that $\chi^\alpha$ is $p$-singular. Then from Lemma \ref{c2} there exists $i\geq 0$ with $b_i(\alpha)\not=a_i$. Let $j$ be maximal such that $b_j(\alpha)\not=a_j$. From Corollary \ref{c3} we have that $j\leq k$ and $b_j(\alpha)<a_j$. From Lemma \ref{l'9} we then have that $j\leq k$ and
\[w_{p^j}(\alpha)=\sum_{i\geq j}p^{i-j}b_i(\alpha)<\sum_{i\geq j}p^{i-j}a_i=\sum_{i=j}^kp^{i-j}a_j.\]
Let now
\[\beta=\left(d_{k,1}p^k,\ldots,d_{k,{h_k}}p^k,\ldots,d_{0,1},\ldots,d_{0,{h_0}}\right)\vdash n\]
be of $p$-adic type. Since
\[\sum_{i=j}^kp^{i-j}\sum_{l=1}^{h_i}d_{i_l}=\sum_{i=j}^kp^{i-j}a_i>w_{p^j}(\alpha)\]
we have from Lemma \ref{p1} that we cannot remove from $\alpha$ a sequence of hooks with lengths given by
\[\left(d_{k,1}p^k,\ldots,d_{k,{h_k}}p^k,\ldots,d_{j,1}p^j,\ldots,d_{j,{h_j}}p^j\right).\]
In particular, by the Murnaghan-Nakayama formula, we have that $\chi^\alpha_\beta=0$.

{\bf ``\ref{i3}$\mathbf{\Rightarrow}$\ref{i2}''}: Clear as $\lambda_n$ is of $p$-adic type.

{\bf ``\ref{i2}$\mathbf{\Rightarrow}$\ref{i4}''}: Assume that we can remove from $\alpha$ a sequence of hooks of lengths given by the parts of $\lambda_n$. We will show that $\chi^\alpha_{\lambda_n}\not=0$.

We have that $\alpha=\alpha_{(p^{k+1})}$ as $n<p^{k+1}$ and that for $j\geq 0$
\[(\alpha_{(p^{j+1})})_{(p^j)}=\alpha_{(p^j)}.\]
Fix now $i\geq 0$ and assume that $\gamma$ is obtained from $\alpha$ by removing a sequence of hooks of lengths $((p^k)^{a_k},\ldots,(p^{i+1})^{a_{i+1}})$. Then $|\gamma|=a_ip^i+\ldots+a_0<p^{i+1}$ and so, as $\gamma$ is obtained from $\alpha$ by removing hooks of lengths divisible by $p^{i+1}$ we have that $\gamma=\alpha_{(p^{i+1})}$. By assumption we can then remove $a_i$ hooks of length $p^i$ from $\alpha_{(p^{i+1})}$. Since $|\alpha_{(p^{i+1})}|<(a_i+1)p^i$ it follows that $w_{p^i}(\alpha_{(p^{i+1})})=a_i$.

From Lemma \ref{p2} we then have that
\begin{eqnarray*}
\chi^\alpha_{\lambda_n}&=&c^\alpha_{p^k}\chi^{\alpha_{(p^k)}}_{((p^{k-1})^{a_{k-1}},\ldots,1^{a_0})}\\
&=&c^{\alpha_{(p^{k+1})}}_{p^k}c^{\alpha_{(p^k)}}_{p^{k-1}}\chi^{\alpha_{(p^{k-1})}}_{((p^{k-2})^{a_{k-2}},\ldots,1^{a_0})}\\
&=&\ldots\\
&=&\chi^{\alpha_{(1)}}_{(0)}\prod_{i=0}^kc^{\alpha_{(p^{i+1})}}_{p^i}\\
&\not=&0,
\end{eqnarray*}
since $\alpha_{(1)}=(0)$ and $\chi^{(0)}_{(0)}=1$ (as $\chi^{(0)}$ is the only irreducible character of $S_0$).

{\bf ``\ref{i4}$\mathbf{\Rightarrow}$\ref{i1}''}: Assume that $\chi^\alpha$ is not $p$-singular. We will prove that we can remove from $\alpha$ a sequence of hooks with lengths given by the parts of $\lambda_n$.

From Lemma \ref{c2} we have that $b_i(\alpha)=a_i$ for $i\geq 0$. Assume that for some $j\geq 0$ we can remove a sequence of hooks of lengths given by the parts of
\[((p^k)^{a_k},\ldots,(p^{j+1})^{a_{j+1}})\]
from $\alpha$. This clearly holds for $j=k$, as in this case $j+1=k+1>k$ and so the above partition does not have any part (of length $>0$). Assume that $\gamma$ can be obtained from $\alpha$ by removing such a sequence of hooks. Since $a_i=b_i(\alpha)=0$ for $i>k$ we have from Lemmas \ref{p1} and \ref{l'9} that
\[w_{p^j}(\gamma)=w_{p^j}(\alpha)-\sum_{i\geq j+1}p^{i-j}a_i=\sum_{i\geq j}p^{i-j}b_i(\alpha)-\sum_{i\geq j+1}p^{i-j}b_i(\alpha)=b_j(\alpha)=a_j.\]
So we can remove from $\gamma$ a sequence of $a_j$ hooks of length $p^j$. In particular we can remove from $\alpha$ a sequence of hooks with lengths given by the parts of
\[((p^k)^{a_k},\ldots,(p^j)^{a_j})\]
and then we can conclude by induction.

Having proved that
\[\ref{i1}\Rightarrow\ref{i3}\Rightarrow\ref{i2}\Rightarrow\ref{i4}\Rightarrow\ref{i1},\]
we have that the 4 conditions are equivalent and so Theorem \ref{t2} is proved.

\section*{Remarks on Navarro's question}

Going back to the question at the beginning of the paper we may reformulate it, due to Corollary \ref{c1}, as follows: are there more $p$-vanishing classes than those labeled by partitions of $p$-adic type?

For $p=2$ or $p=3$ the answer to this question is yes and examples are given by the conjugacy classes labeled by the partitions $(1,1)$ and $(2,1)$ respectively. The author proved in her master thesis (\cite{m1}) that, for $p=2$ or $p=3$, partitions labeling $p$-vanishing conjugacy classes differ from partitions of $p$-adic type only in the parts of length less than 8 or 9 respectively, and that these parts need to correspond to a partition labeling a $p$-vanishing conjugacy class (of $S_m$, where $n=m+lp^i$ with $p^i$ equal to either 8 or 9 and $0\leq m<p^i$ and $l\geq 0$). In \cite{m1} it was also proved that every such conjugacy class is $p$-vanishing. For example $(1,1)$ is a 2-vanishing partition of 2 which is not of 2-adic type. It then follows that $(8,1,1)$ is also a 2-vanishing partition, since $10=8+2$ and $(8)$ is a partition of 8 of 2-adic type. Similarly $(2,1)$ is 3-vanishing but not of 3-adic type. As $21=2\cdot 9+3$ we have that $(18,2,1)$ and $(9,9,2,1)$ are 3-vanishing.

For $p\geq 5$ the question remains open. The author conjectures that the answer is no in this case and work has been done in \cite{m1} to try to prove this conjecture, being able to prove that the conjecture holds as long as, for all $n$, whenever $\alpha\vdash n$ is $p$-vanishing then $\sum_{\alpha_j<a_0}\alpha_j\leq a_0$ (by definition $a_0$ depends on $n$).

For more details about these results see \cite{m3}.

\section*{Acknowledgements}

Most of the work contained in this paper is part of the author's master thesis (\cite{m1}), which was written at the University of Copenhagen, under the supervision of J\o rn B. Olsson, whom the author would like to thank for his help in reviewing the paper.

While writing the paper the author was supported by the DFG grant for the Graduiertenkolleg Experimentelle und konstruktive Algebra at RWTH Aachen University (GRK 1632).

\end{document}